\documentclass[12pt]{article}%
\usepackage{amsmath}%
\usepackage{amsfonts}%
\usepackage{amssymb}%
\usepackage{xfrac}%
\usepackage{graphicx}
\usepackage[all]{xy}
\usepackage{enumitem}
\usepackage{lscape}
\usepackage{tikz}
\usepackage{ytableau}
\usepackage{tikz}
\usepackage{tikz-cd}
\usepackage{geometry}

\providecommand{\U}[1]{\protect\rule{.1in}{.1in}}
\newtheorem{theorem}{Theorem}[section]

\newtheorem{corollary}[theorem]{Corollary}

\newtheorem{definition}[theorem]{Definition}

\newtheorem{lemma}[theorem]{Lemma}

\newtheorem{proposition}[theorem]{Proposition}
\newtheorem{remark}[theorem]{Remark}

\newenvironment{proof}[1][Proof]{\noindent\textbf{#1.} }{\ \rule{0.5em}{0.5em}}

\newcommand{\C}{\mathbb{C}}

\newcommand{\Z}{\mathbb{Z}}
\newcommand{\N}{\mathbb{N}}

\newtheorem{prop}[theorem]{Proposition}

\title{Fusion Rings and Modular Invariance for Affine Lie Algebras: A Double Affine Weyl Group Formulation}

\author{Alejandro Ginory}

\allowdisplaybreaks

\begin{document}

\date{}
\maketitle

\begin{abstract}
	From a certain induced representation $\mathcal{P}_\ell$ of a double affine Weyl group, we construct a ring $\mathcal{F}_\ell$ that is isomorphic to the fusion ring, or Verlinde algebra, associated to affine Lie algebras at fixed positive integer level for both twisted and untwisted type. The induced representation, which also has a natural commutative associative algebra structure and is modular invariant with respect to certain congruence subgroups, contains $\mathcal{F}_\ell$ as an ideal and we show how it naturally inherits the modular invariance property from $\mathcal{P}_\ell.$ This construction directly shows how the action of the modular transformation $S:\tau\to -1/\tau$ determines the structure constants, with respect to a natural basis, of $\mathcal{F}_\ell,$ which are precisely the fusion rules of Verlinde algebras. Using this ideal, we also give a simple proof of a well-known epimorphism between the representation rings of simple Lie algebras and the fusion rings of the corresponding affine Lie algebras.
\end{abstract}

\section{Introduction}

One of the most remarkable features of the representation theory of affine Lie algebras is that the span of the characters of their integrable highest weight modules at fixed positive integral level form modular invariant (with respect to certain congruence subgroups) Frobenius algebras, known as fusion rings or Verlinde algebras \cite{Ver}. The structure constants with respect to the character basis are known as the fusion rules and play a very significant role in conformal field theory \cite{Fuc}. In parallel, the double affine Weyl groups (and, more generally, double affine Hecke algebras) associated with these affine Lie algebras are known to possess modular invariance with respect to the same congruence subgroups \cite{IS}. This leads to the natural question of whether the modular invariance and fusion rules for affine Lie algebras can be realized by double affine Weyl groups.

\medskip

Similar to the way that Weyl groups of finite dimensional simple Lie algebras can be used to describe much of their representation theory, in this work, we show how the double affine Weyl group can be used to describe both the modular invariance properties of characters and the fusion ring structure. We give a natural construction of modular invariant Frobenius algebras $\mathcal{F}_\ell$ that are isomorphic to the fusion rings at level $\ell,$ where $\ell$ is a positive integer, for both twisted and untwisted affine Lie algebras, as an induced representation of the double affine Weyl group. In our construction, a simple argument shows that the modular transformation $S:\tau\mapsto -1/\tau$ associated to $\mathcal{F}_\ell$ gives rise to the fusion rules, which appear as the natural structure constants of $\mathcal{F}_\ell.$ The formula that links the $S$-matrix to the fusion rules is precisely the Verlinde formula. Moreover, $\mathcal{F}_\ell$ appears as a quotient of the representation ring of the corresponding simple Lie algebra, giving a direct proof of the fact that the Kac-Walton algorithm \cite{Wal} for computing fusion rules yields an epimorphism between the representation ring and $\mathcal{F}_\ell.$ Perhaps it is important to mention that we do not investigate conformal blocks nor prove that their dimensions give the structure constants of the fusion algebras of affine Lie algebras. For this link, please see \cite{Bea}, \cite{Tel}, \cite{Ho1}, and the references therein.

\medskip

Let us describe $\mathcal{F}_\ell$ in a bit more detail. Similar to the polynomial representation of double affine Hecke algebras (see \cite{Che}), we first construct an induced $\widetilde{W}$-representation $\mathcal{P}_\ell$ called the truncated polynomial representation, where $\ell$ is a positive integer (the level) and $\widetilde{W}$ is the double affine Weyl group associated to an affine Lie algebra $\mathfrak{g}.$ The truncated polynomial representation is itself an algebra that is isomorphic to the group algebra of a finite abelian group $G_\ell,$ where $G_\ell$ is a quotient of the weight lattice of the associated finite dimensional simple algebra $\overset{\circ}{\mathfrak{g}}$ associated with $\mathfrak{g}.$ In our construction, the $S$ transformation appears as the Fourier transform on the space of functions on $G_\ell.$ It also serves as an intertwiner for a finite Heisenberg subgroup of $\widetilde{W}.$ We then study two distinguished subspaces of the group algebra of $G_\ell$: the subspace of alternating functions and the subalgebra of invariant functions (with respect to the action of the finite Weyl group). In the former case, we define an $S$ matrix, i.e., the matrix given by the Fourier transform with respect to a natural basis of alternating functions, and study some of its properties, see (\ref{defn-S-matrix-alternating}) and prop. \ref{prop-S-symmetric-unitary}. In particular, we define a projective representation of certain congruence subgroups, see prop. \ref{prop-projective-action} and cor. \ref{cor-projective-action-congruence}. In the latter case, we study the Fourier transform, which is closely related to the $S$ matrix from the alternating subspace, and define a principal ideal $\mathcal{F}_\ell$ by using distinguished linear characters of $G_\ell,$ see definition \ref{defn-fusion-principal-ideal}. In the ideal $\mathcal{F}_\ell,$ the images of the characters of $\overset{\circ}{\mathfrak{g}}$ corresponding to level $\ell$ dominant weights form a basis whose structure constants are precisely the fusion rules of $\mathfrak{g}$ at level $\ell,$ see theorem \ref{theorem-fusion-ideal-structure-constants}. The mysterious role of the Verlinde formula in computing the fusion rules, in our setting, is explained to be simply a consequence of the facts that the $S$ matrix is the matrix of the Fourier transform (with respect to the character basis) and that the Fourier transform is an algebra isomorphism between function spaces (taking convolution products to point-wise products). 

\medskip

As mentioned earlier, the algebra $\mathcal{F}_\ell$ is naturally a quotient of the representation ring of $\overset{\circ}{\mathfrak{g}},$ see cor. \ref{cor-homomorphism-rep-ring}, and the corresponding projection homomorphism encodes the Kac-Walton algorithm for computing fusion rules \cite{Wal}. In \cite{Wal}, under the assumption that the fusion rules are diagonalized by the $S$ matrix, it is essentially proved that this projection is a homomorphism. This map was shown to be a homomorphism for certain cases in \cite{Fal} (see also \cite{Bea}) and later, in the context of conformal blocks for untwisted affine Lie algebras, Lie algebra cohomology arguments were used to prove that this map is a homomorphism (in addition to proving the Verlinde conjecture), see \cite{Tel}, \cite{Ku2}. In \cite{Ho1}, Lie algebra cohomology is used to prove the analogous result for both twisted and untwisted affine Lie algebras, also in the context of conformal blocks. A related paper \cite{Ho2} gives a very simple definition of the fusion ring for both twisted and untwisted affine Lie algebras realized as a space of functions on finitely many points. In this paper, our construction shows the link with double affine Weyl groups, which makes clear the modular invariance properties of the fusion rings, and recovers the fusion ring found in \cite{Ho2} as the image of the Fourier transform of $\mathcal{F}_\ell.$ 

\medskip

In our approach, the difference between the twisted and untwisted case is less important since our crucial tool, the Fourier transform, is a map between different function spaces. In the untwisted case, it is tempting to identify both spaces via a natural (yet non-canonical) isomorphism, denoted $\iota_1,$ (in fact, we do so at certain points), but a \emph{different} isomorphism, denoted $\iota_3,$ is used at a key point in constructing the ideal $\mathcal{F}_\ell.$ The function spaces in the twisted case do not have such a natural identification.

\medskip

Let us briefly describe the contents of this paper. In section 2, we go over preliminaries related to groups, affine Lie algebras, congruence subgroups of $SL(2,\Z),$ and characters of affine Lie algebras. In section 3, we introduce the double affine Weyl group and certain related groups that feature prominently in the construction. In section 4, we construct the truncated polynomial representation $\mathcal{P}_\ell.$ Subsection 4.1 explores the modular invariance properties of the subspace of alternating functions. In  subsection 4.2, we construct the fusion principal ideal $\mathcal{F}_\ell$ and prove that it is a quotient of the representation ring of the associated simple Lie algebra. The various maps $\iota_1,$ $\iota_2,$ and $\iota_3$ are defined in section 4 and give identifications of lattice quotients with linear characters. We provide an $A_1^{(1)}$ example at the end of the text.

\medskip

\noindent \textbf{Acknowledgments.} I would like to thank S. Sahi for suggesting this problem and for many helpful conversations. I am also very grateful to S. Kumar and J. Hong for their interest and insightful comments. 

\medskip

\section{Preliminaries}

In this section, we briefly recall basic notions and set notation to be used throughout this paper. The finite group theory content is standard and should be very familiar. For example, content related to harmonic analysis on finite groups can be found in \cite{Ter}. The background related to affine Lie algebras can be found in \cite{Ka}. We follow the notation in that reference very closely. 

\subsection{Group Theory Notions}
\label{sec-group-theory-prelims}
Most of the groups that we will consider here are discrete or finite groups. Let $G$ be a group, then its (complex) \emph{group algebra} is the complex vector space $\C[G]$ with basis given by $\{\sigma\}_{\sigma\in G}$ whose multiplication is defined on the basis by the group multiplication (and extended linearly to the rest of $\C[G]$). If $G$ is abelian with operation $+,$ then the basis of $\C[G]$ is written $\{e^g\}_{g\in G},$ where $e^{g}$ is regarded as a formal exponential and we have $e^{g}\cdot e^{h}=e^{g+h}$ (as the notation suggests). 

\medskip

Let $C(G,\C)$ be the vector space of $\C$-valued functions on $G.$ In the case of finite groups, this space comes with two different associative multiplications: the \emph{convolution product} defined by
\begin{equation}
	f*g(\sigma)=\sum_{\tau\in G}f(\sigma\tau^{-1})g(\tau)\sigma,
\end{equation}
for all $f,g\in C(G,\C),$ and $\sigma\in G,$ and the \emph{point-wise multiplication product} defined by
\begin{equation}
f\cdot g(\sigma)=f(\sigma)\cdot g(\sigma),
\end{equation}
for all $f,g\in C(G,\C),$ and $\sigma\in G.$ Under the convolution product, the linear map defined by $\delta_{\sigma}\mapsto \sigma,$ where $\delta_{\sigma}(\tau)=1$ if $\sigma=\tau$ and $0$ otherwise, is an algebra isomorphism. We will frequently identify $C(G,\C)$ with $\C[G]$ using this isomorphism. Under point-wise multiplication, $C(G,\C)$ is a commutative algebra. 

\medskip

The \emph{Pontryagin dual} of a group $G$ is the abelian group
\begin{equation}
\widehat{G}=\text{Hom}(G,S^1),
\end{equation}
where $S^1$ is the unit circle, under point-wise multiplication. When $G$ is finite and abelian, $\widehat{G}\cong G,$ though the isomorphism is non-canonical. For any group $G,$ define the \emph{Fourier transform} of a function $f:G\to \C$ to be the function $\widehat{f}:\widehat{G}\to \C$ defined by
\begin{align}
\widehat{f}(\chi)=|G|^{-1/2}\sum_{\sigma\in G}f(\sigma)\chi(\sigma)^*,
\end{align}
where $\chi\in \widehat{G}$ and ${}^*$ means complex conjugate. A very important property of the Fourier transform is that for any functions $f,g$ on $G$ and $\chi\in \widehat{G},$
\begin{equation}
\widehat{f*g}(\chi)=\widehat{f}(\chi)\widehat{g}(\chi),
\end{equation}
i.e., the Fourier transform takes convolution products to point-wise products. In fact, when $G$ is finite and abelian, the Fourier transform is an algebra isomorphism $\C[G]\cong C(\widehat{G},\C).$ For finite abelian groups, the \emph{inverse Fourier transform} is defined 
\begin{equation}
	F^\vee(\sigma)=|G|^{-1/2}\sum_{\chi\in \widehat{G}}  \chi(\sigma)F^\vee(\chi),
\end{equation}
for any $F\in C(\widehat{G},\C),$ $\sigma\in G.$

\medskip

\subsection{Affine Lie Algebras}
\label{sec-ALA-prelims}

An affine Kac-Moody algebra (we will also use the term affine Lie algebra) $\mathfrak{g}$ of type $X_N^{(r)}$ is associated to a finite dimensional simple Lie (sub-)algebra $\overset{\circ}{\mathfrak{g}}\subset \mathfrak{g}.$ Let $n$ be the rank of $\overset{\circ}{\mathfrak{g}}.$ Let $a_0,\ldots,a_n$ be the labels of the Dynkin diagram of $\mathfrak{g}$ and $a_0^\vee,\ldots,a_n^\vee$ be the dual labels. The weight lattice $P,$ root lattice $Q,$ and coroot lattice $Q^\vee$ of $\mathfrak{g}$ have their finite counterparts $\overset{\circ}{P},$ $\overset{\circ}{Q},$ and $\overset{\circ}{Q}{}^\vee.$ The Coxeter and dual Coxeter numbers of $\mathfrak{g}$ are denoted $h$ and $h^\vee,$ respectively. We use the usual notation for simple roots and coroots $\alpha_i,\alpha_j^\vee,$ for $j=0,\ldots,n,$ respectively. Let $\mathfrak{h}$ be the Cartan subalgebra of $\mathfrak{g}$ with basis $\alpha_1^\vee,\ldots,\alpha_n^\vee,K,$ and $d$ then denote the projection map onto $\overset{\circ}{\mathfrak{h}},$ the Cartan subalgebra of $\overset{\circ}{\mathfrak{g}},$ by $v\mapsto \overline{v}.$ Here $K$ is the so-called canonical central element. 

\medskip

Let $\Lambda_0,\ldots,\Lambda_n$ be the fundamental weights and $\Lambda_0^\vee,\ldots,\Lambda_n^\vee\in \mathfrak{h}$ be the cofundamental weights, which satisfy
\begin{align}
\label{defn-fund-cofund-weights}
	\langle \Lambda_i,\alpha_j^\vee\rangle =\delta_{ij} \text{ and } \langle \Lambda_i^\vee,\alpha_j\rangle =\delta_{ij} \text{ for }i,j=0,\ldots,n.
\end{align}
Let $P^\vee$ be the lattice spanned by $\{\Lambda_i^\vee\}_{i=0}^{n}$ and $\overset{\circ}{P}{}^\vee$ be the sublattice spanned by $\{\Lambda_i^\vee\}_{i=1}^{n}.$ The sets $P^\vee_\ell,$ $P^{\vee,+},$ etc. have the obvious meanings. One can show that $\langle \Lambda_i^\vee,\alpha_j^\vee\rangle = \frac{a_j}{a_j^\vee}$ and $\langle \Lambda_i,\alpha_j\rangle = \frac{a_j^\vee}{a_j}.$ In particular, for type $X_n^{(r)},$ $r\Lambda_i^\vee\in P$ and $r\alpha_i^\vee\in Q.$ Also define
\begin{align}
	\rho=\Lambda_0+\Lambda_1+\cdots +\Lambda_n \text{ and } \rho^\vee=\Lambda_0^\vee+\Lambda_1^\vee+\cdots +\Lambda_n^\vee 
\end{align}
Note that $\{\lambda+\rho:\lambda\in P^+_\ell\}=P_\ell^{++}$ while an analogous statement holds for $P^\vee_\ell.$ Note that the level of $\rho$ is $h^\vee$ and the level of $\rho^\vee=h.$ 

\medskip

Let $\langle,\rangle$ be the normalized invariant form on $\mathfrak{h}$ and $\nu:\mathfrak{h}\to \mathfrak{h}^*$ be the vector space isomorphism induced by this form. Note that we use the same notation for the canonical pairing between $\mathfrak{h}$ and $\mathfrak{h}^*.$ For any element $\lambda\in \mathfrak{h}^*,$ the \emph{level} $\ell(\lambda)$ is defined to be $\langle\lambda,K\rangle.$ Denote the set of dominant weights by $P^+,$ the set of regular dominant weights by $P^{++},$ the set of the level $\ell$ dominant weights by $P_\ell^+,$ and the set of the level $\ell$ regular dominant weights by $P_{\ell}^{++}.$ For later use, note that
\begin{align}
	Q^\vee &= \overset{\circ}{Q}{}^\vee\oplus \Z K \text{ and}\\
	\nu(Q^\vee) &= \nu(\overset{\circ}{Q}{}^\vee)\oplus \Z \delta,
\end{align}
where $\nu(K)=\delta.$

\medskip

The \emph{Weyl group }$W$ of $\mathfrak{g}$ is the subgroup of the orthogonal group $O(\mathfrak{h}^*,\langle,\rangle)$ generated by the \emph{simple reflections} $s_j,$ for $j=0,\ldots,n,$ on $\mathfrak{h}^*$ defined by 
\begin{align}
	s_j(v)=v-\langle \alpha_j^\vee,v\rangle\alpha_j.
\end{align}
The Weyl group acts on $P_\ell$ and the $P^+_\ell$ is a fundamental domain. The \emph{finite Weyl group} $\overset{\circ}{W}$ is the subgroup of $W$ generated by $s_j,$ for $j=1,\ldots,n.$ For any $w\in W$ or $\overset{\circ}{W},$ its \emph{length} $\ell(w)=\min\{k:w=s_{j_1}\cdots s_{j_k}\}$ where $j_i\in \{0,\ldots,n\},$ $j_i\in \{1,\ldots,n\}$ for $W$ and $\overset{\circ}{W},$ respectively. These groups act on $\mathfrak{h}$ via $\nu.$ We call these actions on $\mathfrak{h}^*$ and $\mathfrak{h}$ the \emph{standard actions}. An explicit decomposition of $W$ can be given as follows. Let $\theta\in \overset{\circ}{Q}$ be the highest root of $\overset{\circ}{\mathfrak{g}},$ then $\delta\in Q$ is the isotropic vector $\delta=a_0\alpha_0+\theta,$ where $a_0$ is the numerical label on the affine node in the Dynkin diagram of $\mathfrak{g}$ (see \cite{Ka} \textsection 4.8). Define the lattice
\begin{align}
	M=\Z\overset{\circ}{W}(a_0^{-1}\theta).
\end{align}
In particular,
\begin{align}
\label{M-cases}
	M=\left\{\begin{matrix}
	\nu(\overset{\circ}{Q}{}^\vee) &\text{ if }r\leq a_0\\
	\overset{\circ}{Q} &\text{ if }r> a_0\\
	\end{matrix}\right. 
\end{align}
and $\overset{\circ}{W}$ acts on $M.$ An important property is that $M\subseteq \nu(\overset{\circ}{Q}{}^\vee).$ Then 
\begin{equation}
	W\cong \overset{\circ}{W}\ltimes M,
\end{equation}
where the element of $W$ corresponding to $\alpha\in M$ is denoted $t_\alpha$ and acts on $\mathfrak{h}^*$ by
\begin{align}
	t_{\alpha}(\lambda)=\lambda +\langle\lambda,K\rangle \alpha - (\langle \alpha,\lambda\rangle + \frac{|\alpha|^2}{2}\langle \lambda,K\rangle)\delta.
\end{align}
Define the \emph{extended Weyl group} to be
\begin{align}
\label{def-ext-aff-Weyl-group}
	W^e=\overset{\circ}{W}\ltimes \overset{\circ}{P}.
\end{align}
For later use, let 
\begin{equation}
\label{M-dual-lattice}
	M^\circ=\{v\in\mathfrak{h}:\langle v,\lambda\rangle\in \Z,\text{ for all }\lambda\in M\},
\end{equation}
and identify it with $\nu(M^\circ)\subseteq \mathfrak{h}^*.$ 

Let $k\in \Z,$ then another important group, acting on $\overset{\circ}{\mathfrak{h}}$ and its dual, is the \emph{affine Weyl group} $W^{aff}_{k}.$ It has a similar decomposition as $W$ in terms of a finite Weyl group and a lattice,
\begin{equation}
	W^{aff}_{k}\cong \overset{\circ}{W}\ltimes kM,
\end{equation}
but its action on $\overset{\circ}{\mathfrak{h}}$ is 
\begin{align}
	wt_{\alpha}(\lambda)=w(\lambda +\alpha)
\end{align}
where $w\in \overset{\circ}{W}$ and $\alpha\in kM.$ The \emph{fundamental alcove} $A_k$ is defined
\begin{align}
\label{fund-alcove}	
	A_k=\{v\in \overset{\circ}{\mathfrak{h}}: \langle v,a_0^{-1}\theta\rangle< k \text{ and } 0<\langle v,\alpha_j\rangle\text{ for }j=1,\ldots,n \}
\end{align}
and its closure is a fundamental domain for $W^{aff}_{k}$ \cite{Hum}. Note that $cl(A_k)\cap \overset{\circ}{P}=\overline{P}_{k}^+,$ where $cl$ denotes closure. 

\medskip

Fix an element $\rho\in\mathfrak{h}^*$ such that $\langle \rho,\alpha_j^\vee\rangle =1$ for $j=0,\ldots,n.$ Define the \emph{dot action} of $W$ to be, for $\lambda,\mu\in \mathfrak{h}^*,$
\begin{align}
	w\cdot \lambda = w(\lambda+\rho)-\rho
\end{align}
and the dot action of $\overset{\circ}{W}$ to be
\begin{align}
	w\cdot \overline{\mu} = w(\overline{\mu}+\overline{\rho})-\overline{\rho}.
\end{align}

\medskip

\subsection{Congruence Subgroups}

We will briefly recall basic definitions involving congruence subgroups. $\text{SL}(2,\Z)$ is the group of matrices with integer coefficients whose determinant is equal to 1. The following subgroups of $\text{SL}(2,\Z)$ are called \emph{principal congruence subgroups}
$$\Gamma_1(r)=\left\{ 
\begin{pmatrix} 
a & b\\ 
c & d 
\end{pmatrix} 
\right|
\left.
\begin{pmatrix} a & b\\ 
c & d
\end{pmatrix}\equiv \begin{pmatrix}
1 & *\\ 
0 & 1
\end{pmatrix} \mod{r} \right\}.$$

\medskip

Define the matrices
$$u_{12}=\begin{pmatrix} 
1 & -1\\ 
0 & 1
\end{pmatrix} \text{ and } u_{21}=\begin{pmatrix} 
1 & 0\\ 
1 & 1 
\end{pmatrix}.$$
For $r=1,2,3,$ $u_{12},u_{21}^r$ generate the groups $\Gamma_1(r)$ and satisfy the $r$-braid relations
\begin{equation}
	\underbrace{u_{12}u_{21}^ru_{12}\cdots }_{b(r)\text{ many}}=\underbrace{u_{21}^r u_{12} u_{21}^r \cdots}_{b(r) \text{ many}}
\end{equation}
where $b(1)=3,b(2)=4,$ and $b(3)=6.$ Note that $\Gamma_1(1)=SL_2(\Z)$ and that it is generated by the matrices
\begin{equation}
	T=u_{12}^{-1}=\begin{pmatrix} 
	1 & 1\\ 
	0 & 1
	\end{pmatrix} \text{ and }
	S=u_{12}u_{21}u_{12}=\begin{pmatrix} 
	0 & -1\\ 
	1 & 0
	\end{pmatrix}.
\end{equation}

\medskip

\subsection{Characters of Affine Lie Algebras}
\label{sec-characters-affine-lie}
Let $\mathfrak{g}$ be a Lie algebra and $V$ be a $\mathfrak{g}$-module, then the character of $V$ is the formal sum \begin{equation}
	\chi_V=\sum_{\lambda}{m_{V}(\lambda)e^\lambda}
\end{equation}
where $V=\bigoplus_{\lambda\in\mathfrak{h}^*}V_{\lambda},$ $V_{\lambda}$ is the weight space of weight $\lambda,$ and $m_{V}(\lambda)=\dim V_\lambda.$ Here we assume that $\dim V_{\lambda}\in \N$ for all $\lambda\in \mathfrak{h}^*.$ The \emph{level} of a highest weight module for an affine Lie algebra is the level of its highest weight.

\medskip

For any symmetrizable Kac-Moody algebra $\mathfrak{g},$ such as the affine Lie algebras and finite dimensional simple Lie algebras, the character of a highest weight integrable module with highest weight $\lambda$ is given by the well-known Kac-Weyl character formula:

\begin{align}
\label{character-formula}
	\chi_{\lambda} = \frac{\sum_{w\in W} (-1)^{\ell(w)}e^{w(\lambda+\rho)}}{\sum_{w\in W} (-1)^{\ell(w)}e^{w(\rho)}}
\end{align}

\noindent where $W$ is the Weyl group of $\mathfrak{g}$, $\ell$ is the length function on $W$, and $\rho\in \mathfrak{h}^*$ is an element such that $\langle \rho,\alpha_i^\vee\rangle =1$ for all $i$. We often write $\varepsilon(w)=(-1)^{\ell(w)}$ for any Weyl group element $w.$ The map $\varepsilon$ is a homomorphism and is called the \emph{sign character}. When $\mathfrak{g}$ is affine and $\chi_{\lambda}$ is the character of an irreducible highest weight $\mathfrak{g}$-module of highest weight $\lambda,$ then $\overset{\circ}{\chi}_{\overline{\lambda}}$ is the character of an irreducible $\overset{\circ}{\mathfrak{g}}$-module of highest weight $\overline{\lambda}.$

\medskip

The representation ring of $\overset{\circ}{\mathfrak{g}}$ is the subalgebra $Ch(\overset{\circ}{\mathfrak{g}})$ of $\C[\overset{\circ}{P}]$ generated by $\overset{\circ}{\chi}_{\overline{\lambda}}$ for $\lambda\in P^+.$ Multiplication of two characters in this ring produces the character of the corresponding tensor product of modules. Later on, we will see the connection between representation rings of finite dimensional simple algebras and the fusion rings constructed here.

\medskip

\subsection{Verlinde Formula}

The Verlinde formula, in one of its several manifestations, relates the action of $SL(2,\Z)$ on the space of level $\ell$ characters of integrable modules for untwisted affine Lie algebras to the fusion rules. The famous formula is
\begin{equation}
\label{verlinde-formula}
	N_{\lambda\mu}^{\nu}=\sum_{\phi\in I}\frac{S_{\lambda\phi}S_{\mu\phi}(S^{-1})_{\phi\nu}}{S_{0\phi}}
\end{equation}
where $S=(S_{\lambda\mu})_{\lambda,\mu\in P_\ell^+}$ (see \cite{Ver}). Notice that replacing $S$ with $SD$ where $D$ is an invertible diagonal matrix will yield the same fusion rules. It is well-known that $S_{\lambda\mu}$ is proportional to 
\begin{equation}
	\sum_{w\in \overset{\circ}{W}}\varepsilon(w)e^{-\frac{2\pi i}{\ell+h^\vee}\langle w(\overline{\lambda}+\overline{\rho}),\overline{\mu}+\overline{\rho}\rangle}
\end{equation}
by a scalar independent of $\lambda$ and $\mu$ \cite{Ka}, \cite{Ver}, \cite{Wal}. 

\medskip

\section{Double Affine Weyl Groups}

Let $\mathfrak{g}$ be an affine Kac-Moody algebra of any type. Recall that the Weyl group $W$ of $\mathfrak{g}$ has the form $W\cong \overset{\circ}{W}\ltimes M$ and that $W$ acts on $Q^\vee\subseteq \mathfrak{h}.$ For any $\alpha\in \mathfrak{h}^*,$ denote by $p_\alpha$ the translation by $\alpha$ operator, i.e., for all $v\in \mathfrak{h}^*,$ $p_\alpha(v)=v+\alpha.$ 

\medskip

\begin{definition}
\label{def-DAWG-EDAWG}
	Let $\mathfrak{g}$ be an affine Lie algebra of any type, then its \emph{double affine Weyl group} (\emph{DAWG}) is the group
	\begin{equation}
	\label{DAWG-decomp1}
		\widetilde{W}=W\ltimes Q^{\vee}.
	\end{equation}
	The \emph{extended double affine Weyl group} (\emph{EDAWG}) is the group
	\begin{equation}
	\label{EDAWG-decomp1}
		\widetilde{W}^{e}=W^{e}\ltimes \left(M^\circ\oplus \frac{1}{m}\Z\delta\right),
	\end{equation}
	(see (\ref{def-ext-aff-Weyl-group}) and (\ref{M-dual-lattice})) where $m$ is the least positive integer such that $m\langle M^\circ,\overset{\circ}{P}\rangle\subseteq \Z.$
\end{definition}

\medskip

As the names suggest, $\widetilde{W}$ is a subgroup of $\widetilde{W}^{e}.$ An element $\beta$ in the second lattice of $\widetilde{W}$ or $\widetilde{W}^{e}$ ($Q^{\vee}$ or $M^\circ\oplus\frac{1}{m}\Z\delta$, respectively) acts by $p_{\beta},$ as defined at the beginning of this section. In particular,
\begin{align}
	t_\alpha p_\beta t_{-\alpha}p_{-\beta}=p_{t_\alpha(\beta)}p_{-\beta}=p_{-\langle\alpha,\beta\rangle\delta}
\end{align}
Note that the DAWG and EDAWG have the form
\begin{align}
	\widetilde{W}\cong \overset{\circ}{W}\ltimes \mathcal{H},\\
	\widetilde{W}^{e}\cong \overset{\circ}{W}\ltimes \mathcal{H}^{e},
\end{align}
where $\mathcal{H}= M\ltimes (\overset{\circ}{Q}{}^\vee \oplus \Z\delta)\cong \overset{\circ}{Q}{}^\vee \ltimes (M\oplus \Z\delta)$ and $\mathcal{H}^{e}= M^\circ\ltimes (\overset{\circ}{P}\oplus \frac{1}{m}\Z\delta)\cong \overset{\circ}{P}\ltimes (M^\circ\oplus \frac{1}{m}\Z\delta)$ are (discrete) Heisenberg groups. We call $\mathcal{H}$ and $\mathcal{H}^{e}$ the \emph{associated Heisenberg group} and \emph{associated extended Heisenberg group}, respectively. The elements of these Heisenberg groups will be written 
$(\alpha,\beta,u)$ where $\alpha\in \overset{\circ}{P},$ $\beta\in M^\circ,$ and $u\in \Z,$ where we identify $(\alpha,0,0)$ with $t_\alpha,$ $(0,\beta,0)$ with $p_\beta,$ and $(0,0,u)$ with $p_{u\delta}=u\delta\in \frac{1}{m}\Z\delta.$ Their multiplication is given by
\begin{equation}
	(\alpha,\beta,u)\cdot (\alpha',\beta',u')=(\alpha+\alpha',\beta+\beta',u+u'+\frac{1}{2}\langle \alpha',\beta\rangle-\frac{1}{2}\langle\alpha,\beta'\rangle)
\end{equation}
and the elements of $\mathcal{H}^e$ act on $\mathfrak{h}^*$ by
\begin{align}
	(\alpha,\beta,u)(\lambda)=t_\alpha p_\beta p_{(u+\frac{1}{2}\langle \alpha,\beta\rangle)\delta}(\lambda)=t_\alpha(\lambda)+\beta+(u-\frac{1}{2}\langle \alpha,\beta\rangle)\delta.
\end{align}
Finally, $\widetilde{W}^{e}$ and $\widetilde{W}$ admit an action by the principal congruence subgroup $\Gamma_1(r)$ defined by
\begin{align}
\label{SL2-Z-action-H}
	(\alpha,\beta,u)\cdot \begin{pmatrix}
	a & b\\
	c & d
	\end{pmatrix}=(a\alpha+c\beta,b\alpha+d\beta,u),
\end{align}
for any $(\alpha,\beta,u)\in \widetilde{W}^{e}_\ell.$ To see this, note that $a\alpha+r\beta\in \overset{\circ}{P}$ for all $\alpha\in \overset{\circ}{P},$ $\beta\in M^\circ$ (see discussion after (\ref{defn-fund-cofund-weights}) and also see (\ref{M-cases}) which relates to the $u_{12}$ action). In the case that $\mathfrak{g}$ is of type $A_{2n}^{(2)},$ $\widetilde{W}^{e}$ and $\widetilde{W}$ admit an action of $\Gamma_1(1)=SL(2,\Z)$ as above, since $M=\nu(\overset{\circ}{Q}{}^\vee).$ 

\medskip

We will also need the following notions. Let $\mathfrak{g}$ be an affine Lie algebra with dual Coxeter number $h^\vee$ and $\ell$ be a positive integer.

\medskip

\begin{definition}
	The finite (extended) Heisenberg group $\mathcal{H}^{e}_\ell$ is
	\begin{align}
		\mathcal{H}^{e}_\ell=\mathcal{H}^{e}/ ((\ell+h^\vee)M\ltimes (\ell+h^\vee)Q^{\vee}).
	\end{align}
	More explicitly,
	\begin{align}
		\mathcal{H}^{e}_\ell=(M^\circ/(\ell+h^\vee)M)\ltimes ((\overset{\circ}{P}\oplus\frac{1}{m}\Z\delta)/(\ell+h^\vee)Q^{\vee}).
	\end{align} 
	We will also denote the elements of $\mathcal{H}^{e}_\ell$ by $(\alpha,\beta,u).$
\end{definition}

\medskip

The group $\overset{\circ}{W}$ acts on $\mathcal{H}^{e}_\ell$ since it acts on $M^\circ,$ $M,$ and $Q^\vee.$ Hence, we can take their semidirect product to form a finite version of the EDAWG. Also note that $\Gamma_1(r)$ acts on $\mathcal{H}^{e}_\ell.$

\medskip

\begin{definition}
	\label{defn-finite-EDAWG}
	The \emph{finite extended double affine Weyl group} is the group
	\begin{equation}
		\widetilde{W}^{e}_\ell=\overset{\circ}{W}\ltimes \mathcal{H}^{e}_\ell.
	\end{equation}
\end{definition}

\medskip

\section{The Induced Representation $\mathcal{P}_\ell$}

Let $\mathfrak{g}$ be an affine Lie algebra of type $X_{N}^{(r)}$ and $\ell$ be a positive integer. Note that in the case $r\leq a_0,$ we have that $M=\nu(\overset{\circ}{Q}{}^\vee)$ and $M^\circ=\overset{\circ}{P}.$ Using the decomposition (\ref{defn-finite-EDAWG}), we consider the subgroup of $\widetilde{W}^{e}_\ell$
\begin{align}
	W_\ell^-&=\{(w,0,\beta,u):\ w\in\overset{\circ}{W}, \beta\in M^\circ/(\ell+h^\vee)\overset{\circ}{Q}{}^\vee,u\in \frac{1}{m}\Z\delta/(\ell+h^\vee)\Z\delta\},
\end{align}
or, equivalently, $W_\ell^-$ is the image of $\overset{\circ}{W}\ltimes 0\ltimes (M^\circ\oplus \frac{1}{m}\Z\delta)$ under the projection map $\widetilde{W}^{e}\to \widetilde{W}^{e}_\ell,$ and also consider the one dimensional $W_\ell^-$-representation $\C_\ell$ where $\overset{\circ}{W}$ and $\{p_\beta\}_{\beta\in M^\circ/(\ell+h^\vee)\overset{\circ}{Q}{}^\vee}$ act trivially and $u\delta\in\frac{1}{m}\Z\delta $ acts by $e^{\frac{2\pi i}{\ell+h^\vee} u}.$

\begin{definition}
\label{defn-trunc-poly-rep}	
	Let $\widetilde{W}^{e}_\ell$ be the finite extended double affine Weyl group and $\C_\ell$ be the $W_\ell^-$-representation defined above, then the \emph{truncated polynomial representation} $\mathcal{P}_\ell$ is the induced representation
	\begin{align}
		\mathcal{P}_\ell=\text{\emph{Ind}}_{W_\ell^-}^{\widetilde{W}^{e}_\ell} \C_\ell.
	\end{align}
\end{definition}

\medskip

Define the finite abelian group
\begin{equation}
	G_\ell = \overset{\circ}{P}/(\ell+h^\vee)M,
\end{equation}
then $\mathcal{P}_\ell\cong \C[G_\ell]$ as vector spaces. From here on, we identify $\mathcal{P}_\ell$ with $\C[G_\ell]$ (and also $C(G_\ell,\C),$ see \textsection\ref{sec-group-theory-prelims}) and give the latter space(s) a $\widetilde{W}^{e}_\ell$-module structure inherited from $\mathcal{P}_\ell.$ When discussing $G_\ell,$ we will often identify an element $\mu\in \overset{\circ}{P}$ with its canonical projection in $G_\ell.$  The action of $\widetilde{W}^{e}_\ell$ on $\C[G_\ell]$ is given by
\begin{align}
	w\cdot e^{\lambda}&=e^{w(\lambda)}\\
	t_\alpha \cdot e^\lambda &= e^{\lambda+\alpha}\\
	p_\beta \cdot e^{\lambda}&= e^{\frac{2\pi i}{\ell+h^\vee} \langle \beta,\lambda\rangle }e^{\lambda},\\
	(u\delta) \cdot e^{\lambda}&= e^{\frac{2\pi i}{\ell+h^\vee} u}e^{\lambda},
\end{align}
for any $w\in \overset{\circ}{W},$ $\alpha,\lambda\in \overset{\circ}{P},$ $\beta\in M^\circ,$ and $u\in \frac{1}{m}\Z.$ The module structure, namely the action of the $t_\alpha,$ coincides with the group algebra multiplication of $\C[G_\ell],$ so we will consider this space as both an algebra and a module. Under the identification discussed in \textsection\ref{sec-group-theory-prelims}, $\widetilde{W}^{e}_\ell$ acts on $C(G_\ell,\C)$ by
\begin{equation}
	(p_\beta t_\alpha w f)(\gamma)=e^{\frac{2\pi i}{\ell+h^\vee}\langle \beta,\gamma\rangle} f(w^{-1}(\gamma-\alpha)),
\end{equation}
for $\alpha,\gamma\in \overset{\circ}{P},$ $\beta\in M^\circ,$ and $w\in \overset{\circ}{W}.$

\medskip

Let us study the Fourier transform on $\C[G_\ell].$ First, we describe the dual in more detail. 

\medskip

\begin{lemma}
	\label{lem-equal-size-quotients}
	For $\mathfrak{g}$ an affine Lie algebra of type $X_{N}^{(r)}$ and $\ell$ be a positive integer, 
	\begin{align}
	\widehat{G}_\ell=\widehat{\overset{\circ}{P}/(\ell+h^\vee)M}\cong M^\circ/(\ell+h^\vee)\overset{\circ}{Q}{}^\vee.
	\end{align}
\end{lemma}

\begin{proof}
	When $r\leq a_0,$ $M=\overset{\circ}{Q}{}^\vee$ and so $M^\circ/(\ell+h^\vee)\overset{\circ}{Q}{}^\vee=\overset{\circ}{P}/(\ell+h^\vee)M.$ When $r>a_0,$ $M=\overset{\circ}{Q}$ and $M^\circ=\overline{P}^\vee=\overset{\circ}{P}{}^\vee$ (see (\ref{defn-fund-cofund-weights})) and so 
	\begin{align}
	M^\circ/(\ell+h^\vee)\overset{\circ}{Q}{}^\vee&=\overset{\circ}{P}{}^\vee/(\ell+h^\vee)\overset{\circ}{Q}{}^\vee\\
	\overset{\circ}{P}/(\ell+h^\vee)M&=\overset{\circ}{P}/(\ell+h^\vee)\overset{\circ}{Q}.
	\end{align}
	In both cases, the homomorphism $\alpha\mapsto e^{\frac{2 \pi i}{\ell+h^\vee}\langle \alpha,\cdot\rangle}\in \widehat{G}_\ell,$ for $\alpha\in M^\circ/(\ell+h^\vee)\overset{\circ}{Q}{}^\vee,$ gives the result.
\end{proof}

\medskip

Define the group homomorphism
\begin{align}
	&\iota_1:M^\circ/(\ell+h^\vee)\overset{\circ}{Q}{}^\vee\to \widehat{G}_\ell\\
	&\iota_1(\alpha)=e^{\frac{2\pi i}{\ell+h^\vee}\langle \alpha,\cdot \rangle}.
\end{align}
It is straightforward to check that this is a well-defined isomorphism. Moreover, this map induces a $\overset{\circ}{W}$-linear isomorphism between $\C[\widehat{G}_\ell]$ and $C(G_\ell,\C)$ which we also denote $\iota_1$ since $\iota_1(w(\alpha))=e^{\frac{2\pi i}{\ell+h^\vee}\langle w(\alpha), \cdot \rangle}=w\cdot \iota_1(\alpha),$ for all $w\in \overset{\circ}{W}$ and $\alpha\in M^\circ.$ In the same vein, we define the map 
\begin{align}
	&\iota_2:\overset{\circ}{P}/(\ell+h^\vee)M\to \widehat{\widehat{G}}_\ell=G_\ell\\
	&\iota_2(\alpha)=e^{-\frac{2\pi i}{\ell+h^\vee}\langle \alpha,\cdot \rangle}
\end{align}
which induces a $\overset{\circ}{W}$-linear isomorphism between $\C[G_\ell]$ and $C(\widehat{G}_\ell,\C).$ Note that $\widetilde{W}^{e}$ also acts on $\C[\widehat{G}_\ell],$ where $\overset{\circ}{W}$ acts the same, $u\delta$ acts by $(u\delta)\cdot e^{\lambda}=e^{\frac{2\pi i}{\ell+h^\vee}u}e^\lambda,$ $\beta\in M^\circ$ acts by $p_\beta\cdot e^{\lambda}=e^{\lambda+\beta},$ and $\alpha\in \overset{\circ}{P}$ acts by $t_\alpha\cdot e^{\lambda}=e^{-\frac{2\pi i}{\ell+h^\vee}\langle\alpha,\lambda\rangle}e^\lambda.$

\medskip

\begin{remark}
	It is worth pointing out that $\iota_1$ gives an identification of $G_\ell$ with $\widehat{G}_\ell$ in the case that $r\leq a_0.$ This identification allows one to consider the Fourier transform (and thus the $S$-matrix as we will soon discuss) as a map on $C(G_\ell,\C).$ It turns out that a different identification, to be made in section \ref{sec-fusion-rules-invariant}, is more useful when describing the fusion ring.
\end{remark}

\medskip

Now, the Fourier transform also commutes with $\overset{\circ}{W}$ since
\begin{align}
	\widehat{we^{\lambda}}(\iota_1(\beta))&=|G_\ell|^{-1/2}e^{-\frac{2\pi i}{\ell+h^\vee}\langle w(\lambda), \beta\rangle}\\
	&=|G_\ell|^{-1/2}e^{-\frac{2\pi i}{\ell+h^\vee}\langle \lambda, w^{-1}(\beta) \rangle}\\
	&=(w\cdot \widehat{e^{\lambda}})(\iota_1(\beta))\\
	&=|G_\ell|^{-1/2}(w\cdot \iota_2(\lambda))(\beta),
\end{align}
for $\beta\in M^\circ,$ $\lambda\in \overset{\circ}{P}$ and $w\in \overset{\circ}{W}.$ More generally, for any $\alpha,\lambda\in\overset{\circ}{P},$  $\beta,\gamma\in M^\circ$ and $w\in \overset{\circ}{W},$ let $f=(\alpha,\beta,u)w\cdot e^\lambda=t_\alpha p_\beta p_{(u+\frac{1}{2}\langle \alpha,\beta\rangle)\delta}w\cdot e^{\lambda},$ then 
\begin{align}
	\widehat{f}(\iota_1(\gamma))&=|G_\ell|^{-1/2}e^{\frac{2\pi i}{\ell+h^\vee}(\langle \beta,w(\lambda)\rangle+u+\frac{1}{2}\langle \alpha,\beta\rangle)}e^{-\frac{2\pi i}{\ell+h^\vee}\langle w(\lambda)+\alpha,\gamma\rangle}\\
	&=|G_\ell|^{-1/2}e^{\frac{2\pi i}{\ell+h^\vee}(u-\frac{1}{2}\langle \alpha,\beta\rangle)}e^{-\frac{2\pi i}{\ell+h^\vee}\langle w(\lambda)+\alpha,\gamma-\beta \rangle}\\
	&=|G_\ell|^{-1/2}e^{-\frac{2\pi i}{\ell+h^\vee}(\langle \alpha,\gamma-\beta\rangle-(u-\frac{1}{2}\langle \alpha,\beta\rangle))}e^{-\frac{2\pi i}{\ell+h^\vee}\langle \lambda,w^{-1}(\gamma-\beta) \rangle}\\
	&=|G_\ell|^{-1/2}e^{-\frac{2\pi i}{\ell+h^\vee}(\langle \alpha,\gamma\rangle-(u+\frac{1}{2}\langle \alpha,\beta\rangle))}e^{-\frac{2\pi i}{\ell+h^\vee}\langle \lambda,w^{-1}(\gamma-\beta) \rangle}\\
	&=(t_\alpha p_\beta p_{(u+\frac{1}{2}\langle \alpha,\beta\rangle)\delta}w |G_\ell|^{-1/2}e^{-\frac{2\pi i}{\ell+h^\vee}\langle \lambda,\cdot \rangle})(\gamma)\\
	&=|G_\ell|^{-1/2}(t_\alpha p_\beta p_{(u+\frac{1}{2}\langle \alpha,\beta\rangle)\delta}w \iota_2(\lambda))(\gamma).
\end{align}
We have proved the following.

\medskip

\begin{prop}
	The Fourier transform is a $\widetilde{W}^{e}$-linear isomorphism between $\C[G_\ell]$ and $\C[\widehat{G}_\ell]$ such that $e^{\lambda}\mapsto |G_\ell|^{-1/2} \iota_2(\lambda).$ 
\end{prop}

\medskip

We will illuminate the close relationship between the Fourier transform and the $S$ transformation. To do so, we must identify the two lattices in $\widetilde{W}^{e}.$ In the case $r>a_0,$ we may identify $\overset{\circ}{P}/(\ell+h^\vee)M$ and $M^\circ/(\ell+h^\vee)\overset{\circ}{Q}{}^\vee$ using the $\Z$-linear map
\begin{align}
\label{varphi-identification-P-Pvee}
	\varphi&:\overset{\circ}{P}\to \overset{\circ}{P}{}^\vee\\
	\varphi(\overline{\Lambda}_j)&=\left\{\begin{matrix}
	\frac{a_j}{a_j^\vee}\overline{\Lambda}_j & \quad \text{ if }X_N^{(r)}\in \{A_{2n-1}^{(2)},D_{n+1}^{(2)}\}\\
	\frac{a_{n+1-j}}{a_{n+1-j}^\vee}\overline{\Lambda}_{n+1-j} & \text{if }X_N^{(r)}\in \{E_{6}^{(2)},D_{4}^{(3)}\}
	\end{matrix}
	\right.
\end{align}
for $j=1,\ldots,n.$ (Note that $\varphi$ relates the root system of $\mathfrak{g}$ to its \emph{adjacent} Lie algebra, which has the same dual Coxeter number (see \cite{Ka}).) In the case that $r\leq a_0,$ $\overset{\circ}{P}/(\ell+h^\vee)M=M^\circ/(\ell+h^\vee)\overset{\circ}{Q}{}^\vee,$ we take $\varphi=Id.$ Using these maps, identify $\C[G_\ell]$ and $\C[\widehat{G}_\ell],$ then for any $\alpha,\beta\in \overset{\circ}{P},$ 
\begin{align}
	\widehat{(\alpha,\beta,0)\cdot e^\lambda}&=\sum_{\gamma\in \overset{\circ}{P}/(\ell+h^\vee)M} |G_\ell|^{-1/2}e^{-\frac{2\pi i}{\ell+h^\vee}(\frac{1}{2}\langle \alpha,\varphi(\beta )\rangle+\langle \lambda+\alpha,\varphi(\gamma)-\varphi(\beta )\rangle)}e^{\gamma}\\
	&=p_{-\alpha}t_{\beta}(-\frac{1}{2}\langle \alpha,\varphi(\beta)\rangle)\delta\sum_{\gamma\in \overset{\circ}{P}/(\ell+h^\vee)M} |G_\ell|^{-1/2}e^{-\frac{2\pi i}{\ell+h^\vee}\langle \lambda,\varphi(\gamma)\rangle}e^{\gamma}\\
	&=t_{\beta}p_{-\alpha}(\frac{1}{2}\langle \alpha,\varphi(\beta)\rangle)\delta\sum_{\gamma\in \overset{\circ}{P}/(\ell+h^\vee)M} |G_\ell|^{-1/2}e^{-\frac{2\pi i}{\ell+h^\vee}\langle \lambda,\varphi(\gamma)\rangle}e^{\gamma}\\
	&=(\beta,-\alpha,-\frac{1}{2}\langle\alpha,\varphi(\beta)\rangle)\widehat{e^{\lambda}}\\
	&=((\alpha,\beta,-\langle\alpha,\varphi(\beta)\rangle/2)\cdot S)\widehat{e^{\lambda}}.
\end{align}

\medskip

If we identify $\C[G_\ell]$ with $\C[\widehat{G}_\ell]$ as above, then we may define an action of $S$ and $T$ on $\C[G_\ell]$ by
\begin{align}
	S\cdot e^{\lambda}&=((\lambda,0,0)\cdot S)\cdot \widehat{e^0}=(0,-\lambda,0)\cdot \widehat{e^0}=|G_\ell|^{-1/2}\sum_{\mu\in G_\ell}e^{-\frac{2\pi i}{\ell+h^\vee}\langle \lambda,\varphi(\mu)\rangle}e^{\mu}\\
	T\cdot e^{\mu}&=((\mu,0,0)\cdot T)\cdot e^0=(\mu,\mu,0)\cdot e^0=e^{\frac{\pi i|\mu|^2}{\ell+h^\vee}}e^\mu.
\end{align}
A priori, this action is a projective action of $SL(2,\Z)$ on $\C[G_\ell]$ and is known as the (projective) \emph{Weil representation} (arising from the associated finite Heisenberg). We delay discussing its linearization for the next section.

\begin{remark}
	The Fourier transform takes functions on $G_\ell$ to functions on $\widehat{G}_\ell.$ In the case that $r\leq a_0,$ these spaces of functions are the same, except that one is equipped with the convolution product and the other with the point-wise product. Therefore, there is a Weil representation on that space. In the case that $r>a_0,$ the Fourier transform maps between \emph{different} function spaces, also with different algebra structures. The identification made above was certainly not canonical.
\end{remark}

\medskip

\subsection{Modular Invariance and $\overset{\circ}{W}$-Alternating Functions}
\label{sec-mod-inv-alternating-space}
Let us use the extended double affine Weyl group framework to study the subspace $\mathcal{A}_\ell\subseteq \mathcal{P}_\ell$ of alternating functions that is intimately related to characters of the affine Lie algebra $\mathfrak{g}$ of type $X_N^{(r)}.$ We show that this subspace is invariant under action of the congruence subgroup $\Gamma_1(r)$ and compute relevant matrices, like the $S$ and $T$ matrices when $r=1,$ with respect to a natural basis. At the end, we construct a projective representation of $\Gamma_1(r)$ on the space spanned by level $\ell$ characters of $\mathfrak{g}$ and show its close similarity to the linear modular group action arising from theta functions.

\medskip

\begin{proposition}
\label{prop-projective-action}
	Let $\mathfrak{g}$ be an affine Lie algebra of type $X_N^{(r)},$ $\ell$ a positive integer, and $\mathcal{P}_\ell$ its truncated polynomial representation. Then $\mathcal{P}_\ell$ is a projective $\Gamma_1(r)\ltimes \widetilde{W}^{e}_\ell$ representation. Moreover, the action of $\Gamma_1(r)$ on $\mathcal{P}_\ell$ commutes with the action of $\overset{\circ}{W}.$
\end{proposition}

\begin{proof}
	$\mathcal{P}_\ell$ is an irreducible $\widetilde{W}^{e}_\ell$ representation and so, a standard argument (see for example [GHH, \textsection 1.3]) shows that $\mathcal{P}_\ell$ admits a projective action of $\Gamma_1(r)\ltimes \widetilde{W}^{e}_\ell,$ i.e., there is a homomorphism $\Gamma_1(r)\to \text{PGL}(\mathcal{P}_\ell).$ The action of $\Gamma_1(r)$ commutes with that of $\overset{\circ}{W}$ because the action of $\Gamma_1(r)$ on $\widetilde{W}^{e}$ leaves $\overset{\circ}{W}$-invariant.
\end{proof}

\medskip

\begin{corollary}
\label{cor-projective-action-congruence}
	Let $\mathcal{P}_\ell=\bigoplus_{\lambda} n_\lambda V_\lambda,$ where $V_{\lambda}$ is an irreducible $\overset{\circ}{W}$-module corresponding to $\lambda$ and $\lambda$ ranges over isomorphism classes of irreducible $\overset{\circ}{W}$-modules, then, for each $\lambda,$ $n_\lambda V_\lambda$ is invariant under the action of $\Gamma_1(r).$ In particular, $\Gamma_1(r)$ stabilizes $\mathcal{A}_\ell$ and $\mathcal{P}_\ell^{\overset{\circ}{W}}.$ 
\end{corollary}

\medskip

Much has been written about lifting projective representations of $SL(2,\Z)$ (and symplectic groups more generally), see for example \cite{Wei}, \cite{Sha}, \cite{GH}, \cite{LV}. Our situation is better than the general case since the action of $\Gamma_1(r)$ factors through a finite group, i.e., the kernel of the action is contained in $SL(2,r\Z/(\ell+h^\vee)\Z).$ For example, when $r=1$ and $\ell+h^\vee$ is not a multiple of four it is a general fact that these representations can be made linear \cite{Bey}. Indeed, our projective action can be always lifted to a linear representation of $SL(2,\Z)$ by using theta functions (see \cite{KP}). The question of linearizing this representation fully in terms of the double affine Weyl group is beyond the scope of this paper and we hope to address it in the future. For the purposes of describing the fusion ring, the projective action is sufficient.

\medskip

Consider the subspace $\mathcal{A}_\ell\subseteq \C(G_\ell,\C)$ of functions $f$ satisfying
\begin{align}
	(wf)(\alpha)=\varepsilon(\alpha)f(\alpha)
\end{align}
for all $w\in \overset{\circ}{W}$ (recall that $\varepsilon$ is the sign character \textsection \ref{sec-characters-affine-lie}) As elements in $\C[G_\ell],$ any group algebra element can be projected to $\mathcal{A}_\ell$ by applying the operator
\begin{align}
	\mathcal{E}=\frac{1}{|\overset{\circ}{W}|}\sum_{w\in \overset{\circ}{W}}\varepsilon(w)w.
\end{align}
Suppose that $\mu\in \mathfrak{h}^*,$ such that $\overline{\mu}\in \overset{\circ}{P},$ is dominant and that $s_i(\mu)=\mu$ for some $i=0,1,\ldots,N,$ then it is easy to see that $\mathcal{E}(e^{\overline{\mu}})=0$ in $\C[G_\ell].$ It follows that the group algebra elements
\begin{align}
\label{alternants-defn}
	A_{\overline{\lambda}}=|\overset{\circ}{W}|\mathcal{E}\cdot e^{\overline{\lambda}}=\sum_{w\in \overset{\circ}{W}}\varepsilon(w)e^{w(\overline{\lambda})}
\end{align}
for $\lambda\in P_{\ell+h^\vee}^{++}$ span $\mathcal{A}_\ell.$ Indeed, $\{A_{\overline{\lambda}+\overline{\rho}}\}_{\lambda\in P^{+}_\ell}$ forms a basis of this space (see \textsection\ref{sec-ALA-prelims}). 

\medskip

\begin{proposition}
\label{prop-alternating-linear-independence}
	When $r\leq a_0,$ the set of functions $\{\widehat{A}_{\overline{\lambda}+\overline{\rho}}\}_{\lambda\in P^{+}_\ell}$ is linearly independent as functions on $\{\iota_1(\overline{\mu}+\overline{\rho})\}_{\mu\in P_{\ell}^{+}}.$  When $r>a_0,$ the set of functions $\{\widehat{A}_{\overline{\lambda}+\overline{\rho}}\}_{\lambda\in P^{+}_\ell}$ is linearly independent as functions on $\{\iota_1(\overline{\mu}+\overline{\rho}^\vee)\}_{\mu\in P_{\ell}^{\vee+}}.$ Moreover, when $r>a_0,$ we have $|P^{+}_\ell|=|P_{\ell}^{\vee+}|.$ 
\end{proposition}

\begin{proof}
	To start, since (the closures of) $P_{\ell+h^\vee}^{++}$ and $P_{\ell+h^\vee}^{\vee++}$ are fundamental domains for $P_{\ell+h^\vee}$ and $P_{\ell+h^\vee}^\vee$ when $r\leq a_0$ and $r>a_0,$ respectively, $\{A_{\overline{\lambda}+\overline{\rho}}\}_{\lambda\in P^{+}_\ell}$ is linearly independent in $\C[G_\ell].$ It follows that $\{\widehat{A}_{\overline{\lambda}+\overline{\rho}}\}_{\lambda\in P^{+}_\ell}$ is linearly independent as functions on $M^\circ/(\ell+h^\vee)\overset{\circ}{Q}{}^\vee.$
	
	Consider the case $r>a_0.$ For any $\alpha\in \mathfrak{h}^*$ such that $\overline{\alpha}\in M^\circ$ and $\lambda\in P^{+}_\ell,$ we have that 
	\begin{align}
	\label{alternating-evals}
		\widehat{A}_{\overline{\lambda}+\overline{\rho}}(\iota_1(\overline{\alpha}))=|G_\ell|^{-1/2}\sum_{w\in \overset{\circ}{W}}\varepsilon(w)e^{-\frac{2\pi i}{\ell+h^\vee}\langle w(\overline{\lambda}+\overline{\rho}),\overline{\alpha}\rangle}.
	\end{align}
	We may act on $\overline{\alpha}$ by the (affine) Weyl group $\overset{\circ}{W}\ltimes (\ell+h^\vee)\overset{\circ}{Q}{}^\vee$ (which may be identified with the Weyl group of the adjacent algebra when $r>a_0$) and this will only change the value of $\widehat{A}_{\overline{\lambda}+\overline{\rho}}(\iota_1(\overline{\alpha}))$ by a sign at most. Therefore, $\{\widehat{A}_{\overline{\lambda}+\overline{\rho}}\}_{\lambda\in P^{+}_\ell}$ is linearly independent as functions on $P_{\ell+h^\vee}^{\vee+}.$ But the quantity (\ref{alternating-evals}) vanishes if $s_i(\alpha)=\alpha$ for any $i=0,1,\ldots,n,$ where $s_i$ are the generators (see \textsection \ref{sec-ALA-prelims}). This proves the linear independence of $\{\widehat{A}_{\overline{\lambda}+\overline{\rho}}\}_{\lambda\in P^{+}_\ell}$ on the set $\{\iota_1(\overline{\mu}+\overline{\rho}^\vee)\}_{\mu\in P_{\ell}^{\vee+}}.$ A similar, yet slightly simpler, argument proves the statement in the case that $r\leq a_0.$
	
	To prove the last statement, recall the map $\varphi$ defined in (\ref{varphi-identification-P-Pvee}), and first note that for all $\mu\in P^{+}_\ell,$ $\varphi(\overline{\mu}+\overline{\rho})=\overline{\lambda}$ for some $\lambda\in P_{\ell+h^\vee}^{\vee++}.$ Second, note that for any $\lambda\in P_{\ell}^{\vee+},$ $\lambda+\rho^\vee$ has level $\ell+h^\vee$ and $\varphi^{-1}(\overline{\lambda}+\overline{\rho}^\vee)=\varphi^{-1}(\overline{\lambda})+\overline{\rho}=\overline{\mu}+\overline{\rho}$ for some $\mu\in P_\ell^+.$
\end{proof}

\bigskip

Using the result above, we can explicitly write the Fourier transform 
\begin{align}
	\widehat{A}_{\overline{\lambda}+\overline{\rho}}=\left\{\begin{matrix}
	|G_\ell|^{-1/2}\sum_{{\overline{\mu}\in P_{\ell}^{+}}}\widehat{A}_{\overline{\lambda}+\overline{\rho}}(\iota_1(\overline{\mu}+\overline{\rho}))A_{\overline{\mu}+\overline{\rho}} & \text{if }r\leq a_0 \\
	|G_\ell|^{-1/2}\sum_{{\overline{\mu}\in P_{\ell}^{\vee+}}}\widehat{A}_{\overline{\lambda}+\overline{\rho}}(\iota_1(\overline{\mu}+\overline{\rho}^\vee))A_{\overline{\mu}+\overline{\rho}^\vee} & \text{if }r> a_0
	\end{matrix}
	\right. ,
\end{align}
where $\overline{\lambda}\in P_{\ell}^{+}.$ Let us denote these quantities by
\begin{align}
	a_{\lambda\mu}=\widehat{A}_{\overline{\lambda}+\overline{\rho}}(\iota_1(\overline{\mu}+\overline{\rho})),
\end{align}
where $\lambda\in P_{\ell}^{+}$ and $\mu\in P_{\ell}^{+},$ and the analogous quantities in the case that $r>a_0.$ 

\medskip

Define the matrices
\begin{align}
\label{defn-S-matrix-alternating}
	T_{\mathcal{A}}=(\delta_{\lambda,\mu}e^{\frac{\pi i|\overline{\lambda}+\overline{\rho}|^2}{\ell+h^\vee}})_{\lambda,\mu\in P_\ell^+} \text{ and } S_{\mathcal{A}}=(a_{\lambda\mu})_{\lambda,\mu},
\end{align}
where $\lambda,\mu$ range over the sets described at the end of the previous paragraph. A direct computation involving the orthogonality of characters proves the following proposition. (The proof is left to the reader.)

\medskip

\begin{proposition}
\label{prop-S-symmetric-unitary}
	If we identify $P_{\ell}^{+}$ and $P_{\ell}^{\vee+}$ using the map $\varphi$ defined in (\ref{varphi-identification-P-Pvee}), then $S_{\mathcal{A}}$ is a symmetric, unitary matrix.
\end{proposition}

In the case that $r\leq a_0,$ under the correspondence $\chi_\lambda\mapsto A_{\overline{\lambda}+\overline{\rho}},$ the matrices $T_{\mathcal{A}},S_{\mathcal{A}}$ gives a projective action of $\Gamma_1(1)=SL(2,\Z),$ arising \emph{naturally} from the internal symmetries of the extended double affine Weyl group, on the space spanned by the level $\ell$ characters of $\mathfrak{g}.$ This action can in fact be made into a linear representation by suitably renormalizing the matrices $T_{\mathcal{A}},S_{\mathcal{A}}.$ In \cite{KP}, the authors define a linear representation of $SL(2,\Z)$ by expressing the characters as quotients of linear combinations of theta functions. In particular, their matrices $S$ and $T$ differ from the matrices $S_{\mathcal{A}}$ and $T_{\mathcal{A}}$ by an overall factor of $i^{|\overset{\circ}{\Delta}_+|}$ for the former and by an overall factor of $e^{-\frac{\pi i|\overline{\rho}|^2}{h^\vee}}$ for the latter.

In the case that $r>a_0,$ one may identify the $\overset{\circ}{W}$-alternating subspace of $C(\widehat{G}_\ell,\C)$ with the span of level $\ell$ characters of the adjacent Lie algebra (see the comment after (\ref{varphi-identification-P-Pvee})). The matrix $S_{\mathcal{A}}$ is then interpreted as a map between the space spanned by level $\ell$ characters of $\mathfrak{g}$ and the space spanned by level $\ell$ characters of its adjacent algebra.

\medskip

Unfortunately, this story does not shed much light on the fusion product structure on characters. We will see that the linear independence of the Fourier transforms of the alternating functions is crucial in defining the fusion ring. We address this in the next section.

\medskip

\subsection{Fusion Rules and $\overset{\circ}{W}$-Invariant Functions}
\label{sec-fusion-rules-invariant}

Now we modify $\iota_1$ to yield a different isomorphism between $M^\circ/(\ell+h^\vee)\overset{\circ}{Q}{}^\vee$ and $\widehat{G}_\ell.$ Define the map
\begin{align}
	&\iota_3:M^\circ/(\ell+h^\vee)\overset{\circ}{Q}{}^\vee\to \widehat{G}_\ell\\
	&\iota_3(\alpha)=\left\{\begin{matrix}
	e^{\frac{2\pi i}{\ell+h^\vee}\langle \alpha+\overline{\rho},\cdot \rangle} &\text{if }r\leq a_0\\
	e^{\frac{2\pi i}{\ell+h^\vee}\langle \alpha+\overline{\rho}^\vee,\cdot \rangle} &\text{if }r>a_0
	\end{matrix}
	\right.
\end{align}
It is easy to see that the set of functions $\{\iota_3(\alpha)\}_{\alpha\in M^\circ/(\ell+h^\vee)\overset{\circ}{Q}{}^\vee}$ is equal to $\widehat{G}_\ell.$ 

\medskip

Note that $\iota_3(\alpha)\iota_3(\beta)\neq \iota_3(\alpha+\beta),$ i.e., unlike $\iota_1,$ the map $\iota_3$ is \emph{not }a group homomorphism. On the other hand, when $r\leq a_0,$ $\overset{\circ}{W}$ acts on $\widehat{G}_\ell$ by
\begin{align}
	(w\iota_3(\alpha))(e^\mu)&=e^{\frac{2\pi i}{\ell+h^\vee}\langle \alpha+\overline{\rho},w^{-1}(\mu)\rangle}\\
	&=e^{\frac{2\pi i}{\ell+h^\vee}\langle w(\alpha+\overline{\rho}),\mu\rangle}\\
	&=\iota_3(w\cdot \alpha)(e^\mu),
\end{align}
where $\cdot$ in the last line is the dot action on $G_\ell$ (see \textsection \ref{sec-ALA-prelims}). A similar action holds when $r>a_0,$ except with $\overline{\rho}^\vee$ playing the role of $\overline{\rho}.$ The preceding discussion proves the following.

\begin{proposition}
	The map $\iota_3:M^\circ/(\ell+h^\vee)\overset{\circ}{Q}{}^\vee\to \widehat{G}_\ell$ is a bijection that intertwines the dot action of $\overset{\circ}{W}$ on $M^\circ/(\ell+h^\vee)\overset{\circ}{Q}{}^\vee$ and the dual action (with respect to the standard action) of $\overset{\circ}{W}$ on $\widehat{G}_\ell.$
\end{proposition}

\medskip

Since $G_{\ell}$ is abelian, we can write its character table as the $g\times g$ matrix
\begin{align}
\label{Psi-character-table-G-ell}
\Psi=\begin{pmatrix}
\iota_3(\mu_1)(e^{\lambda_1}) & \iota_3(\mu_2)(e^{\lambda_1}) & \cdots & \iota_3(\mu_g)(e^{\lambda_1})\\
\iota_3(\mu_1)(e^{\lambda_2}) & \iota_3(\mu_2)(e^{\lambda_2}) & \cdots & \iota_3(\mu_g)(e^{\lambda_2})\\
\vdots & & \ddots & \vdots\\
\iota_3(\mu_1)(e^{\lambda_g}) & \iota_3(\mu_2)(e^{\lambda_g}) & \cdots & \iota_3(\mu_g)(e^{\lambda_g})
\end{pmatrix}
\end{align}
where $g=|G_\ell|,$ $\{\lambda_1,\ldots,\lambda_g\}\subset \overset{\circ}{P}$ is a complete set of coset representatives of $\overset{\circ}{P}/(\ell+h^\vee)M,$ and $\{\mu_1,\ldots,\mu_g\}\subset M^\circ$ is a complete set of coset representatives of $M^\circ/(\ell+h^\vee)\overset{\circ}{Q}{}^\vee.$ For later use, assume that $\{\lambda_1,\ldots,\lambda_p\}$ and $\{\mu_1,\ldots,\mu_p\}$ correspond to the sets of dominant integral weights of level $\ell.$ Since $\iota_3(\mu_j)$ is a function on $G_\ell,$ we identify $\iota_3(\mu_j)$ with the element in $\C[G_\ell]$ defined by
\begin{equation}
	\psi(\mu_j)=\sum_{k=1}^{g}\iota_3(\mu_j)(e^{\lambda_k})e^{\lambda_k}=\sum_{k=1}^{g}e^{\frac{2\pi i}{\ell+h^\vee}\langle \mu_j+\overline{\rho},\lambda_k\rangle}e^{\lambda_k},
\end{equation}
then its Fourier transform satisfies 
\begin{equation}
	\widehat{\psi(\mu_j)}(\iota_3(\mu_i))=g^{1/2}\delta_{ij}.
\end{equation}
It follows that $g^{-1/2}\widehat{\psi(\mu_j)}$ are orthogonal idempotents that span $C(\widehat{G}_\ell,\C)$ (under point-wise multiplication). For any $f\in \C[G_\ell],$ we will express its Fourier transform as an element of $\C[\widehat{G}_\ell]\cong C(\widehat{G}_\ell,\C)$ in the form
\begin{equation}
	\widehat{f}=\sum_{j=1}^{g}\widehat{f}(\iota_3(\mu_j))g^{-1/2}\widehat{\psi(\mu_j)}.
\end{equation}
We will never have occasion to use the convolution product of $\C[\widehat{G}_\ell],$ instead we will use the point-wise multiplication from $C(\widehat{G}_\ell,\C),$ i.e., for $f,g\in \C[\widehat{G}_\ell],$
\begin{equation}
	\widehat{f}\cdot \widehat{g}=\sum_{j=1}^{g}\widehat{f}(\iota_3(\mu_j))\widehat{g}(\iota_3(\mu_j))g^{-1/2}\widehat{\psi(\mu_j)}.
\end{equation}

\medskip

Suppose that $f\in \C[G_\ell]$ is $\overset{\circ}{W}$-invariant (with respect to the standard action) and $r\leq a_0,$ then its Fourier transform satisfies
\begin{align}
	\widehat{f}(\iota_3(\mu_j))&=g^{-1/2}\sum_{k=1}^{g}f(\lambda_k)e^{-\frac{2\pi i}{\ell+h^\vee}\langle \lambda_k,\mu_j+\overline{\rho}\rangle}\\
	&=g^{-1/2}\sum_{k=1}^{g}f(\lambda_k)e^{-\frac{2\pi i}{\ell+h^\vee}\langle \lambda_k,w(\mu_j+\overline{\rho})\rangle}\\
	&=\widehat{f}(\iota_3(w\cdot \mu_j)),
\end{align}
for all $w\in \overset{\circ}{W}.$ The analogous statement holds for $r>a_0.$ It follows that $\widehat{f}$ is constant on the $\overset{\circ}{W}$-orbits of $M^\circ/(\ell+h^\vee)\overset{\circ}{Q}{}^\vee$ (or, rather, the image under $\iota_3$) under the dot action. 

\medskip

Let $\overset{\circ}{\chi}_{\lambda}$ be the character of the irreducible finite dimensional representation $V_{\overline{\lambda}}$ of $\overset{\circ}{\mathfrak{g}}$ with highest weight $\overline{\lambda}.$ We call these $\overset{\circ}{\mathfrak{g}}$-\emph{characters} to distinguish them from linear characters of $G_\ell.$ Suppose that $\lambda\in P^+,$ then consider the element
\begin{equation}
	\overset{\circ}{\chi}_\lambda =\overset{\circ}{\chi}_{\overline{\lambda}} = \sum_{\sigma\in \overset{\circ}{P}} m_{\overline{\lambda}}(\sigma) e^{\sigma}\in \C[\overset{\circ}{P}],
\end{equation}
where $m_{\overline{\lambda}}(\sigma)$ is the multiplicity of the weight space of $\sigma$ in $V_{\overline{\lambda}}.$ When $\lambda$ has level $\ell,$ identify $\overset{\circ}{\chi}_\lambda$ with its canonical image in $\C[G_\ell].$ Note that $\overset{\circ}{\chi}_{\ell\overline{\Lambda}_0}=\overset{\circ}{\chi}_{0}$ is identity in $\C[G_\ell]$ and $\overset{\circ}{\chi}_\lambda\in \C[G_\ell]$ is invariant under the standard action of $\overset{\circ}{W}$ in $\C[G_\ell]$ for all $\lambda\in P_\ell^+.$

\medskip

When $r\leq a_0,$ the Fourier transform of these $\overset{\circ}{\mathfrak{g}}$-characters is given by
\begin{align}
	\widehat{\overset{\circ}{\chi}}_\lambda(\iota_3(\mu_j))&=g^{-1/2}\sum_{k=1}^{g} m_{\overline{\lambda}}(\lambda_k)e^{-\frac{2\pi i}{\ell+h^\vee}\langle \lambda_k,\mu_j+\overline{\rho}\rangle}\\
	&=g^{-1/2}\frac{\sum_{w\in \overset{\circ}{W}} \varepsilon(w)e^{-\frac{2\pi i}{\ell+h^\vee}\langle w(\overline{\lambda}+\overline{\rho}),\mu_j+\overline{\rho}\rangle}  }{\sum_{w\in \overset{\circ}{W}} \varepsilon(w)e^{-\frac{2\pi i}{\ell+h^\vee}\langle w(\overline{\rho}),\mu_j+\overline{\rho}\rangle} }
\end{align}
where $\mu_j\in M^\circ/(\ell+h^\vee)\overset{\circ}{Q}{}^\vee$ (where the second equality is valid when $j=1,\ldots,p$). When $r>a_0,$ we have the same formula except that $\mu_j+\overline{\rho}$ is replaced by $\mu_j+\overline{\rho}^\vee.$ We write $\overset{\circ}{\chi}_\lambda(\hat{\mu}_j)$ instead of $\widehat{\overset{\circ}{\chi}}_\lambda(\iota_3(\mu_j)).$ The $\overset{\circ}{W}$-invariance of $\overset{\circ}{\chi}_\lambda$ implies that the $\overset{\circ}{\mathfrak{g}}$-character is constant (when evaluated) on the dot action $\overset{\circ}{W}$-orbits of $M^\circ/(\ell+h^\vee)\overset{\circ}{Q}{}^\vee.$ 

\medskip

\begin{remark}
\label{remark-linear-ind-Lie-characters}
	It is crucial to note that the denominator and numerator vanish whenever $\mu_j+\overline{\rho}$ (or $\mu_j+\overline{\rho}^\vee$) is not a regular dominant weight. In fact, whenever the denominator vanishes, so does the numerator. Nonetheless, by the results of section \textsection\ref{sec-mod-inv-alternating-space}, the level $\ell$ $\overset{\circ}{\mathfrak{g}}$-characters are linearly independent as functions over the $\overset{\circ}{W}$-orbits of elements $\mu_j+\overline{\rho},$ where $\mu_j\in M^\circ$ is the projection of a dominant element of level $\ell.$ In particular, $A_{\overline{\rho}}$ does not vanish at any of these points.
\end{remark}

\medskip

The Fourier transform of $\overset{\circ}{\chi}_\lambda$ can be written
\begin{equation}
	\widehat{\overset{\circ}{\chi}}_\lambda=\sum_{j=1}^{g}\overset{\circ}{\chi}_\lambda(\hat{\mu}_j)g^{-1/2}\widehat{\psi(\mu_j)}.
\end{equation}
Recall that $\{\mu_1,\ldots,\mu_p\}$ corresponds to the set dominant integral weights of level $\ell$ in $M^\circ.$ Since the $\overset{\circ}{\mathfrak{g}}$-characters are invariant under the standard action of $\overset{\circ}{W},$ we can write
\begin{equation}
	\widehat{\overset{\circ}{\chi}}_\lambda=\sum_{\mu}\overset{\circ}{\chi}_\lambda(\hat{\mu})\phi^{\mu}
\end{equation}
where the sum is taken over representatives of the orbits of the dot action of $\overset{\circ}{W}$ on $M^\circ/(\ell+h^\vee)\overset{\circ}{Q}{}^\vee$ and 
\begin{equation}
	\phi^{\mu}=\sum_{\nu\in \overset{\circ}{W}\cdot \mu}g^{-1/2}\widehat{\psi(\nu)}.
\end{equation}
Note that we can take the representatives to be $\{\nu_1,\ldots,\nu_p,\nu_{p+1},\ldots\}$ where $\nu_j=\mu_j$ for $j=1,\ldots,p$ and $\{\nu_j\}_{j\geq 1}\subseteq \{\mu_j\}_{j=1}^{g}.$ Note that $\phi^{\nu_j}$ are also orthogonal idempotents in $\C[\widehat{G}_\ell]$ and so their inverse Fourier transforms $(\phi^{\nu_j})^\vee$ are orthogonal idempotents in $\C[G_\ell].$ 

\medskip

We now construct an ideal in $\C[G_\ell]$ that is isomorphic as an algebra with the associated fusion ring of $\mathfrak{g},$ i.e., with structure constants given by the fusion rules (\ref{verlinde-formula}). Denote by $\Delta_\ell$ the function on $G_\ell$ given by
\begin{equation}
	\Delta_\ell=(\phi^{\mu_1}+\phi^{\mu_2}+\cdots + \phi^{\mu_p})^\vee,
\end{equation}
which we treat as an element of $\C[G_\ell]$ as usual. It is worth making the point that $\widehat{\Delta}_\ell$ is 1 on $\nu_1,\ldots,\nu_p$ and 0 on $\nu_{p+1},\ldots.$ 

\medskip

\begin{remark}
	The definition of $\Delta_\ell$ is motivated by the fact that $A_{\overline{\rho}}$ acts as zero on representations with character $\nu_{p+1},\ldots,$ and so, alternately, we could define $\Delta_\ell$ as the unique element of $\C[G_\ell]$ that acts as identity on all representations of $G_\ell$ on which $A_{\overline{\rho}}$ acts nontrivially and which acts as zero on all representations of $G_\ell$ on which $A_{\overline{\rho}}$ acts as zero.
\end{remark}

\medskip

Let $\lambda,\mu\in P_\ell^+,$ then
\begin{equation}
	\widehat{\overset{\circ}{\chi}_\lambda \overset{\circ}{\chi}_\mu}=\widehat{\overset{\circ}{\chi}}_\lambda\widehat{\overset{\circ}{\chi}}_\mu=\sum_{j= 1}^{p}\overset{\circ}{\chi}_\lambda(\hat{\nu_j})\overset{\circ}{\chi}_\mu(\hat{\nu_j})\phi^{\nu_j}.
\end{equation}
Since the $\overset{\circ}{\mathfrak{g}}$-characters are linearly independent as functions on the $\overset{\circ}{W}$ orbits (with respect to the dot action) of dominant weights $\nu_1,\ldots,\nu_p$ (using the $\iota_3$ map, also see prop. \ref{prop-alternating-linear-independence} and remark \ref{remark-linear-ind-Lie-characters}), it follows that there are unique linear combinations of $\overset{\circ}{\mathfrak{g}}$-characters such that, for $\nu\in \{\nu_j\}_{j=1}^{p}$
\begin{equation}
	\phi^{\nu}+\sum_{l\geq p+1}b'_{\nu\nu_l}\phi^{\nu_l}=\sum_{\lambda\in P_\ell^+}b_{\nu\lambda}\widehat{\overset{\circ}{\chi}}_\lambda,
\end{equation}
where $b_{\nu\lambda},b'_{\nu\nu_l}\in \C.$ Clearly, the $b_{\nu\lambda}$ are entries of the inverse matrix of $\chi=(\overset{\circ}{\chi}_\lambda(\hat{\mu}))_{\mu,\lambda\in P_\ell^+}$ when $r\leq a_0$ or $\chi=(\overset{\circ}{\chi}_\lambda(\hat{\mu}))_{\mu\in P_\ell^{\vee+},\lambda\in P_\ell^+}$ when $r>a_0.$ Note that $\chi_{\mu\lambda}A_{\overline{\rho}}(\iota_3(\overline{\mu}_j))=A_{\lambda\mu}.$ It follows that
\begin{align}
	\widehat{\overset{\circ}{\chi}_\lambda \overset{\circ}{\chi}_\mu}&=\widehat{\overset{\circ}{\chi}}_\lambda\widehat{\overset{\circ}{\chi}}_\mu=\sum_{\nu\in P_\ell^+}c_{\lambda\mu}^{\nu}\widehat{\overset{\circ}{\chi}}_\nu+\sum_{j\geq p+1}d_{\lambda\mu}^{j}\phi^{\nu_j}\\
	\widehat{\Delta}_\ell\widehat{\overset{\circ}{\chi}_\lambda \overset{\circ}{\chi}_\mu}&=\widehat{\Delta_\ell \overset{\circ}{\chi}}_\lambda\widehat{\Delta_\ell \overset{\circ}{\chi}}_\mu=\sum_{\nu\in P_\ell^+}c_{\lambda\mu}^{\nu}\widehat{\Delta_\ell\overset{\circ}{\chi}}_\nu.
\end{align}
In order to determine the coefficients $c_{\lambda\mu}^\nu,$ for $\lambda,\mu,\nu\in P_\ell^+,$ we only need to compute 
\begin{equation}
	c_{\lambda\mu}^{\nu}=\sum_{\gamma} \frac{\chi_{\lambda\gamma}\chi_{\mu\gamma}(\chi^{-1})_{\gamma\nu}}{\chi_{0 \gamma }}.
\end{equation}
where $\gamma\in P_\ell^+$ if $r\leq a_0$ and $\gamma\in P_\ell^{\vee+}$ if $r> a_0.$ By (\ref{verlinde-formula}) and the surrounding discussion, it follows that $c_{\lambda\mu}^{\nu}=N_{\lambda\mu}^{\nu}.$ The above discussion motivates the following definition and proves the following theorem.

\medskip

\begin{definition}
\label{defn-fusion-principal-ideal}
	The \emph{fusion principal ideal} of $\C[G_\ell]$ is the principal ideal $\mathcal{F}_\ell = \C[G_\ell]\Delta_\ell.$
\end{definition}

\medskip

\begin{theorem}
\label{theorem-fusion-ideal-structure-constants}
	The fusion principal ideal $\mathcal{F}_\ell$ has linear basis $\Delta_\ell\overset{\circ}{\chi}_{\lambda},$ for $\lambda\in P_\ell^+,$ and the corresponding structure constants are the fusion rules $N_{\lambda\mu}^{\nu}.$
\end{theorem}

The following result is well-known, see \cite{Bea}, \cite{Ku1}, etc., but the construction of $\mathcal{F}_\ell$ and the preceding theorem provides a quick direct proof.

\begin{corollary}
\label{cor-homomorphism-rep-ring}
	Let $Ch_\ell(\mathfrak{g})$ be the fusion ring attached to $\mathfrak{g}$ at level $\ell.$ The map $F:Ch(\overset{\circ}{\mathfrak{g}})\to Ch_\ell(\mathfrak{g})$ defined
	\begin{equation}
		F(\overset{\circ}{\chi}_{\lambda})=\left\{\begin{matrix}
		\varepsilon(w_\lambda)\chi_{w_\lambda\cdot \lambda} & \text{if for some }w\in W^{aff}_{\ell+h^\vee},w\cdot \lambda\in A_{\ell+h^\vee}\\
		0&\text{otherwise}
		\end{matrix}		
		\right.
	\end{equation}
	where $w_\lambda\in W^{aff}_{\ell+h^\vee}$ is the unique element such that $w_\lambda\cdot \lambda$ lies in the fundamental alcove $A_{\ell+h^\vee},$ is a homomorphism. In particular, the fusion rules are integers.
\end{corollary}

\begin{proof}
	Denote by $\pi:\C[\overset{\circ}{P}]\to \C[G_\ell]$ the canonical projection map. Recall that $Ch(\overset{\circ}{\mathfrak{g}})=\C[\overset{\circ}{P}]^{\overset{\circ}{W}}$ and consider the restriction map $\pi|_{Ch(\overset{\circ}{\mathfrak{g}})}.$ The map $\C[G_\ell]\to\mathcal{F}_\ell$ given by $f\mapsto f\Delta_\ell$ is an algebra homomorphism, therefore composing the two maps gives a homomorphism $Ch(\overset{\circ}{\mathfrak{g}})\to \mathcal{F}_\ell.$ The previous theorem shows that $Ch_\ell(\mathfrak{g})\cong \mathcal{F}_\ell.$ 
	
	In $\C[\overset{\circ}{P}],$ consider the elements
	\begin{align}
		A_{\overline{\lambda}}&=\sum_{w\in \overset{\circ}{W}}\varepsilon(w)e^{w(\overline{\lambda})},
	\end{align}
	such that $\lambda\in P,$ where we use the same notation as for their projection in $\C[G_\ell].$ (This is justified since $e^{w(\overline{\lambda}+\overline{\rho})}=e^{w'(\overline{\lambda}+\overline{\rho})}$ if and only if $\overline{\lambda}+\overline{\rho}-w^{-1}w'(\overline{\lambda}+\overline{\rho})\in (\ell+h^\vee)M$ which is equivalent to $w=w'.$) Let $\overset{\circ}{\chi}_{\overline{\lambda}}\in Ch(\overset{\circ}{\mathfrak{g}}),$ then $\overset{\circ}{\chi}_{\overline{\lambda}}A_{\overline{\rho}}=A_{\overline{\lambda}+\overline{\rho}}.$ Suppose that $w_\lambda\in W^{aff}_{\ell+h^\vee}$ such that $w_\lambda(\overline{\lambda}+\overline{\rho})\in cl(A_{\ell+h^\vee})$ (see \ref{fund-alcove}), then the projection of $A_{\overline{\lambda}+\overline{\rho}}\in \C[\overset{\circ}{P}]$ is precisely  $\varepsilon(w_{\lambda})A_{w_\lambda(\overline{\lambda}+\overline{\rho})}\in \C[G_\ell].$ It follows that 
	\begin{align}
		\pi(\overset{\circ}{\chi}_{\lambda}A_{\overline{\rho}})=\pi(\overset{\circ}{\chi}_{\lambda})A_{\overline{\rho}}=\varepsilon(w_{\lambda})\overset{\circ}{\chi}_{w_{\lambda}\cdot \lambda}A_{\overline{\rho}}=\varepsilon(w_{\lambda})A_{w_\lambda(\lambda+\overline{\rho})}\in \C[G_\ell].
	\end{align} 
	(Note that the projection is 0 whenever $w_\lambda(\overline{\lambda}+\overline{\rho})\notin A_{\ell+h^\vee}.$) Hence 
	$$(\pi(\overset{\circ}{\chi}_{\lambda})-\varepsilon(w_{\lambda})\overset{\circ}{\chi}_{w_{\lambda}\cdot \lambda})A_{\overline{\rho}}=0.$$
	After taking the Fourier transform, this implies that
	$$\widehat{\pi(\overset{\circ}{\chi}_{\lambda})}(\iota_3(\mu_j))-\varepsilon(w_{\lambda})\widehat{\overset{\circ}{\chi}}_{w_{\lambda}\cdot \lambda}(\iota_3(\mu_j))=0$$ 
	for $j=1,\ldots,p$ (recall that $\widehat{\overset{\circ}{\chi}}_{w_{\lambda}\cdot \lambda}(\iota_3(\mu_j))$ is the evaluation of the character $\overset{\circ}{\chi}_{\lambda}$ at $e^{\frac{2\pi i}{\ell+h^\vee}\langle \mu_j+\overline{\rho},\cdot \rangle}$ or $e^{\frac{2\pi i}{\ell+h^\vee}\langle \mu_j+\overline{\rho}^\vee,\cdot \rangle},$ depending on the case, also see remark \ref{remark-linear-ind-Lie-characters} related to the vanishing $A_{\overline{\rho}}$). It follows that $$\pi(\overset{\circ}{\chi}_{\lambda})\Delta_{\ell}=\varepsilon(w_{\lambda})\overset{\circ}{\chi}_{w_{\lambda}\cdot \lambda}\Delta_{\ell}$$
	since their Fourier transforms are equal, as desired.
	
	For the last statement, note that the product of two $\overset{\circ}{\mathfrak{g}}$-characters in $\C[\overset{\circ}{P}]$ is the $\overset{\circ}{\mathfrak{g}}$-character of the corresponding tensor product, which decomposes as an integral linear combination of $\overset{\circ}{\mathfrak{g}}$-characters.
\end{proof}

\medskip

For convenience, let us identify $\overset{\circ}{\chi}_{\overline{\lambda}}$ with $\Delta_\ell \overset{\circ}{\chi}_{\overline{\lambda}}$ in $\mathcal{F}_\ell.$ The fusion principal ideal $\mathcal{F}_\ell$ admits the bilinear form
\begin{align}
	\langle f,g\rangle_0=\sum_{\mu} \widehat{fg}(\iota_3(\overline{\mu}))=\sum_{\mu} \widehat{f}(\iota_3(\overline{\mu}))\widehat{g}(\iota_3(\overline{\mu}))
\end{align}
where $f,g\in \mathcal{F}_\ell$ and $\mu$ ranges over $P_\ell^{+}$ when $r\leq a_0$ and $P_\ell^{\vee+}$ when $r>a_0.$ It is nondegenerate since $(\phi^{\mu})^\vee$ form an orthogonal basis, where $\mu_j$ ranges over the sets specified in the previous sentence. It also satisfies
\begin{align}
	\langle fg,h\rangle_0 = \langle f,gh\rangle_0
\end{align}
for any $f,g,h\in\mathcal{F}_\ell$ due to the fact that the Fourier transform is a homomorphism. Now, recall that $D\chi=S_{\mathcal{A}}$ where $D=(\delta_{i,j}A_{\overline{\rho}}(\iota_3(\overline{\mu}_j)))_{1\leq i,j\leq p},$ and by prop (\ref{prop-S-symmetric-unitary}), $S_{\mathcal{A}}$ is unitary, we may also define the form 
\begin{align}
	\langle f,g\rangle_1=\sum_{\mu} \widehat{f}(\iota_1(\overline{\mu}))\widehat{g}(\iota_1(\overline{\mu}))^*\widehat{A}_{\overline{\rho}}(\iota_1(\overline{\mu}))^2
\end{align}
where ${}^*$ denotes complex conjugation and $\mu$ ranges over $P_{\ell+h^\vee}^{++}$ when $r\leq a_0$ and $P_{\ell+h^\vee}^{\vee++}$ when $r>a_0.$ Under the form $\langle,\rangle_1,$ the elements $\{\overset{\circ}{\chi}_{\overline{\lambda}}\}_{\lambda\in P_\ell^{+}}$ form an orthonormal basis, see \cite{Bea} and \cite{Ho2} for a similar construction. Also note that
\begin{equation}
	\langle \overset{\circ}{\chi}_{\overline{\lambda}}\overset{\circ}{\chi}_{\overline{\mu}},\overset{\circ}{\chi}_{\overline{\nu}}\rangle_1=\langle \overset{\circ}{\chi}_{\overline{\lambda}},\overset{\circ}{\chi}_{-w_0(\overline{\mu})}\overset{\circ}{\chi}_{\overline{\nu}}\rangle_1
\end{equation}
where $w_0$ is the longest element in $\overset{\circ}{W}.$ The above forms immediately imply the semisimplicity of $\mathcal{F}_\ell$ and also give a Frobenius algebra structure. 

\medskip

\subsubsection{$A_1^{(1)}$ Example}

As an example, let us consider $\mathfrak{g}$ to be of type $A_1^{(1)}$ and level $\ell.$ The relevant characters $\overset{\circ}{\chi}_n$ are indexed by natural numbers $0\leq n\leq \ell$ and can be written
\begin{equation}
\overset{\circ}{\chi}_n=e^{-n}+e^{-n+2}+\cdots +e^{n-2}+e^{n}.
\end{equation}
Let $\psi^n=\psi(n\overline{\Lambda}_1).$ Let us denote the elements of $G_\ell$ by $0,1,\ldots,2\ell+3,$ where $k$ corresponds to $k\overline{\Lambda}_1,$ then we have
\begin{align}
\psi^n&=\sum_{k=0}^{2\ell+3}\psi^{n}(e^{k})e^{k}=\sum_{k=0}^{2\ell+3}e^{\frac{2\pi i}{\ell+h^\vee}\langle k,n+1\rangle}e^{k}\\
&=\sum_{k=0}^{2\ell+3}e^{\frac{\pi i}{\ell+2}k(n+1)}e^{k}\\
\end{align}
Now, suppose that $\ell=1.$ The fusion ring multiplication table in this case is given below.
\begin{center}
	\begin{tabular}{|c ||c|c|} 
		\hline
		 & $\overset{\circ}{\chi}_0$ & $\overset{\circ}{\chi}_1$  \\ 
		\hline\hline
		$\overset{\circ}{\chi}_0$ & $\overset{\circ}{\chi}_0$ &$\overset{\circ}{\chi}_1$ \\ 
		\hline
		$\overset{\circ}{\chi}_1$ & $\overset{\circ}{\chi}_1$ & $\overset{\circ}{\chi}_0$\\
		\hline 
	\end{tabular}
\end{center}
In $\mathcal{P}_\ell,$ we have
\begin{align}
\psi^0&=1+e^{\frac{\pi i}{3}}e^{1}+e^{\frac{2\pi i}{3}}e^{2}+e^{\frac{3\pi i}{3}}e^{3}+e^{\frac{4\pi i}{3}}e^{4}+e^{\frac{5\pi i}{3}}e^{5}\\
\psi^1&=1+e^{\frac{2\pi i}{3}}e^{1}+e^{\frac{4\pi i}{3}}e^{2}+e^{\frac{6\pi i}{3}}e^{3}+e^{\frac{8\pi i}{3}}e^{4}+e^{\frac{10\pi i}{3}}e^{5}\\
\psi^2&=1+e^{\frac{3\pi i}{3}}e^{1}+e^{\frac{6\pi i}{3}}e^{2}+e^{\frac{9\pi i}{3}}e^{3}+e^{\frac{12\pi i}{3}}e^{4}+e^{\frac{15\pi i}{3}}e^{5}\\
\psi^3&=1+e^{\frac{4\pi i}{3}}e^{1}+e^{\frac{8\pi i}{3}}e^{2}+e^{\frac{12\pi i}{3}}e^{3}+e^{\frac{16\pi i}{3}}e^{4}+e^{\frac{20\pi i}{3}}e^{5}\\
\psi^4&=1+e^{\frac{5\pi i}{3}}e^{1}+e^{\frac{10\pi i}{3}}e^{2}+e^{\frac{15\pi i}{3}}e^{3}+e^{\frac{20\pi i}{3}}e^{4}+e^{\frac{25\pi i}{3}}e^{5}\\
\psi^5&=1+e^{\frac{6\pi i}{3}}e^{1}+e^{\frac{12\pi i}{3}}e^{2}+e^{\frac{18\pi i}{3}}e^{3}+e^{\frac{24\pi i}{3}}e^{4}+e^{\frac{30\pi i}{3}}e^{5}\\
\end{align}
Also, we see that
\begin{align}
\phi^0&=\frac{1}{\sqrt{6}}(\widehat{\psi}^0+\widehat{\psi}^{4}),\ \phi^1=\frac{1}{\sqrt{6}}(\widehat{\psi}^{1}+\widehat{\psi}^{3}),\ \phi^2=\frac{1}{\sqrt{6}}\widehat{\psi}^{2},\text{ and } \phi^5=\frac{1}{\sqrt{6}}\widehat{\psi}^5,
\end{align}
and so $\phi^0,\phi^1$ correspond to regular dominant weights, while $\phi^2,\phi^5$ correspond to the non-regular dominant weights $(\ell+h^\vee)\theta=3\overline{\Lambda}_1$ and $0,$ respectively. Hence $\widehat{\Delta}_1=\phi^0+\phi^1.$ Now, consider that
\begin{align}
	\overset{\circ}{\chi}_0&=e^{0}\\
	\widehat{\overset{\circ}{\chi}}_0&=\phi^0+\phi^1+\phi^2+\phi^5\\
	\overset{\circ}{\chi}_1&=e^{-1}+e^{1}\\
	\widehat{\overset{\circ}{\chi}}_1&=(e^{-\frac{\pi i}{3}}+e^{\frac{\pi i}{3}})\phi^0+(e^{-\frac{2\pi i}{3}}+e^{\frac{2\pi i}{3}})\phi^1-2\phi^2+2\phi^5\\
	&=\phi^0-\phi^1-2\phi^2+2\phi^5.
\end{align}
The Fourier transform of the products are
\begin{align}
	\widehat{\overset{\circ}{\chi}_0\overset{\circ}{\chi}_j}&=\widehat{\overset{\circ}{\chi}_j}\\
	\widehat{\overset{\circ}{\chi}_1\overset{\circ}{\chi}_1}&=\phi^0+\phi^1+4\phi^2+4\phi^5\\
	&=\widehat{\overset{\circ}{\chi}_0} +3\phi^2+3\phi^5
\end{align}
which implies that $\Delta_1\overset{\circ}{\chi}_0\cdot \Delta_1\overset{\circ}{\chi}_j=\Delta_1\overset{\circ}{\chi}_j$ and $(\Delta_1\overset{\circ}{\chi}_1)^2=\Delta_1\overset{\circ}{\chi}_0,$ matching the multiplication table of the fusion product.

\end{document}